\documentclass[12pt]{article}

\usepackage{amsfonts,amssymb,latexsym,amsmath,amscd,tikz}
\usepackage{mathrsfs}
\usepackage{times}
\usepackage{anysize}

\input xy
\xyoption{all}
\usepackage[all,knot]{xy}

\marginsize{1in}{1in}{1in}{1in}

\newcommand{\CC}{\mathbb{C}}

\newcommand{\QQ}{\mathbb{Q}}

\newcommand{\ZZ}{\mathbb{Z}}

\newcommand{\kK}{\mathcal{K}}

\newtheorem{lemma}{Lemma}[section]
\newtheorem{conjecture}[lemma]{Conjecture}
\newtheorem{corollary}[lemma]{Corollary}
\newtheorem{theorem}[lemma]{Theorem}
\newtheorem{definition}[lemma]{Definition}
\newtheorem{example}[lemma]{Example}

\newtheorem{remark}[lemma]{Remark}

\newenvironment{proof}{{\bf
     Proof.}}{$\blacksquare$ \vspace{2mm}}

\DeclareMathOperator{\Kh}{Kh}
\DeclareMathOperator{\Coef}{Coef}
\DeclareMathOperator{\MF}{MF}
\DeclareMathOperator{\HHH}{HH}

\newcommand{\Sl}{\mathfrak{sl}}

\title{On stable Khovanov homology of torus knots}
\author{E. Gorsky,  A. Oblomkov, J. Rasmussen}
\date{}

\begin{document}
\maketitle

\begin{abstract}
We conjecture that the 
stable Khovanov homology of torus knots can be described  as the Koszul homology
of an explicit non-regular sequence of quadratic polynomials.
The corresponding Poincar\'e series turns out to be related to the Rogers-Ramanujan identity.
\end{abstract}




\section{Introduction}

In \cite{kho} Khovanov constructed a knot homology theory which categorifies the Jones polynomial using a combinatorial construction in terms of a knot projection. Following the early computations of Bar-Natan and Shumakovitch \cite{bnat1,katlas,KhoHo}, it became evident that the torus knots \(T(n,m)\)  had ``interesting" Khovanov homology, in the sense that their homology was much larger than might have been guessed from looking at the corresponding Jones polynomial, had torsion of high order, {\it etc.} Further advances in computation, most notably Bar-Natan's geometric Khovanov homology \cite{bnat2}, have enabled us to calculate Khovanov homology of torus knots up through \(T(7,n)\), where \(n\) is relatively large \cite{katlas,shucomm}. These calculations have tended to confirm our first impression of overall complexity.

Nevertheless, there are indications that the Khovanov homology of torus knots is not only interesting, but may be important to our understanding of Khovanov homology as a whole. The first result in this direction is the theorem of Sto\v si\'c \cite{stosic}, who proved that if we fix \(n\) and allow \(m\) to vary, then (after a suitable renormalization), the groups \(\Kh(T(n,m))\) tend to a well-defined limit, which we denote by \(\Kh(T(n,\infty))\). More recently, Rozansky \cite{rozansky} has shown that 
the Khovanov complex of the infinite torus braid provides a categorified version of the Jones-Wenzl projector, and thus should play an important role in the theory of colored Khovanov homology 
\cite{fss,gw,webster}.
 In this framework, \(\Kh(T(n,\infty))\) appears as the \(n\)-colored Khovanov homology of the unknot.

In this paper, we consider a conjectural description of \(\Kh(T(n,\infty))\) for all \(n\):

\begin{conjecture}
\label{conj:main}
The unreduced stable Khovanov homology \(\Kh(T(n,\infty))\)  is dual to the homology of the differential graded algebra generated by even variables $x_0,\ldots,x_{n-1}$ and odd variables $\xi_0,\ldots,\xi_{n-1},$ equipped with the differential $d_2$ defined by 
$$d_2(\xi_{k})=\sum_{i=0}^{k}x_{i}x_{k-i} \quad \text{and}  \quad d_2(x_k) = 0.$$
Equivalently, this is the 
 Koszul complex determined by the (nonregular) sequence  \(d_2(\xi_k)\) for \(k=0,\ldots, n-1\). 

\end{conjecture}

\begin{remark}
The homology of the chain complex described in the conjecture should be  the \(\Sl(2)\) Khovanov-Rozansky homology \cite{khoro1} of  \(T(n,\infty)\). This, in turn, is dual to the ordinary Khovanov homology (in the usual sense that they are homologies of dual chain complexes.) 
\end{remark}

Khovanov homology is bigraded; it is equipped with both {\it polynomial} (\(q\)) and {\it homological} (\(t\)) gradings. With respect to the usual normalizations for these gradings, 
the generators \(x_k\) and \(\xi_k\)  are graded as follows:
$$\deg(x_k)=q^{2k+2}t^{2k},\quad \deg(\xi_{k})=q^{2k+4}t^{2k+1}.$$
The differential \(d_2\) preserves the $q$-grading and lowers the $t$-grading by $1$. 

\begin{definition}
We will denote the homology of $d_2$ by $\Kh_{alg}(n,\infty)$.
\end{definition}

Conjecture~\ref{conj:main} arose in our work with Shende \cite{GORS} on the relation between the HOMFLY-PT  homology \cite{khoro2} of torus knots and the representation theory of the rational Cherednik algebra. More specifically, it is known \cite{dgr,khdiff} that one can pass from the HOMFLY-PT homology of a knot \(K\) to its \(\Sl(N)\) Khovanov-Rozansky homology  \cite{khoro1} by means of a spectral sequence. The main conjecture of \cite{GORS} relates the HOMFLY-PT homology of \(T(n,m)\) to the representation theory of the rational Cherednik algebra. On the right-hand side of this equivalence, it is possible to construct certain natural differentials which we believe should correspond to the differentials needed to pass to the \(\Sl(N)\) homology. We arrived at the conjecture above by computing these representation-theoretic differentials for $N=2$ in the limiting case \(m \to \infty\).

\begin{remark}
General considerations about the HOMFLY-PT homology and the differentials on it suggest that 
\(\Kh(T(n,\infty))\) should be the homology of 
\(\ZZ[x_0, \ldots, x_n] \otimes \Lambda^*[\xi_0,\ldots, \xi_n]\) with respect to a differential \(d_2'\) of the form
$$ d_2'(\xi_i) = \sum_{i=0}^k\alpha_{ik} x_i x_{k-i}$$
for some  \(\alpha_{ik}\). The content of the calculation in \cite{GORS} is that  all \(\alpha_{ik}\) should be equal to \(1\). 
\end{remark}

The  first goal of the paper is  to summarize the computational evidence supporting Conjecture~\ref{conj:main}. 
In light of the remark, it is important to check that the conjecture predicts properties of the Khovanov homology which would not be predicted by \(d_2'\) with a generic choice of \(\alpha_{ik}\). In section~ \ref{sec:evidence} and the appendices, we give examples where this is the case using both homology with coefficients in \(\ZZ/p\) and homology with coefficients in \(\QQ\). 

Our second goal is to  investigate the underlying structure of \(\Kh_{alg}\). 
This homology is especially simple if we use \(\ZZ/2\) coefficients:
\begin{theorem}
\(\Kh_{alg}(n,\infty;\ZZ/2)\) has the following Poincar\'e series:
$$P_{n}(q,t;\mathbb{Z}_2)=\prod_{i=0}^{n-1}\frac{(1+q^{2i+4}t^{2i+1})}{ (1-q^{2i+2}t^{2i})}\prod_{i=0}^{\lfloor\frac{n-1}{2}\rfloor}\frac{(1-q^{4i+4}t^{4i})}{ (1+q^{4i+4}t^{4i+1})}.$$
\end{theorem}


With rational coefficients, the homology is more complicated. In section~\ref{Sec:mus}, we construct some explicit elements of \(\Kh_{alg}(n,\infty;\QQ)\),
 as well as some relations which they satisfy. This leads to the following
 \begin{conjecture}
 \label{conj:presentation}
 As an algebra over \(\QQ\), \(\Kh_{alg}(n,\infty)\) has a presentation with $n$ even generators \(x_0,\ldots, x_{n-1}\) and $(n-1)$ odd generators \(\mu_0,\ldots,\mu_{n-2}\) (where \(\mu_{i}\) has bidegree \(q^{2i+8}t^{2i+3}\)) and relations of the form 
 \begin{equation}
x(z)^2=0,\quad x(z)\mu(z)=0,\quad \ddot{x}(z)\mu(z)-\dot{x}(z)\dot{\mu}(z)=0,\quad \mu(z)\dot{\mu}(z)=0
\end{equation}
where $ x(z) = \sum_{i=0}^{n-1} x_iz^i$  $\mu(z) = \sum_{i=0}^{n-2} \mu_i z^i$, and each equation above is to be interpreted modulo \(z^n\). 
 \end{conjecture}
 
Following ideas of Feigin and  Stoyanovsky (\cite{fs}, see also \cite{lf,loktev}), we derive  a conjectural explicit formula for the Poincar\'e polynomial of \(\Kh(T(n,\infty))\).  Feigin and Stoyanovsky studied the structure of the coinvariants for the integrable representation of $\widehat{\Sl(2)}$ at level 1 using the vertex operator equations of Lepowsky and Primc (\cite{lp}), which turn out to be analogous to our Koszul differential. 
The resulting stable homology is described by the equation (\ref{pnintro}) for the unreduced theory, and by the equation (\ref{pnred}) for the reduced theory. 
They generalize the result of Feigin and Stoyanovsky, which is itself a generalization of the Rogers-Ramanujan identity (\cite{andrews}).

\begin{conjecture}
The Poincar{\'e} series of \(\Kh_{alg}(n,\infty;\QQ) \) can be expressed as
\begin{equation}
\label{pnintro}
P_{n}(q,t)=\frac{1}{ \prod_{k=1}^{n}(1-q^{2k}t^{2k-2})}\sum_{p=0}^{\infty}(-1)^{p}\prod_{k=1}^{p}(1-q^{2k}t^{2k-2})\times \\
\end{equation}
$$\times\prod_{k=3p+1}^{n-1}(1+q^{2k+6}t^{2k+1})\prod_{k=1}^{2p-1}(1+q^{2k+2}t^{2k-1})\times$$
$$[q^{5p^2+p}t^{5p^2-3p}(1+\chi^+_p q^{6p+4}t^{6p-1})(1+q^{6p+6}t^{6p+1})\binom{n-2p+1}{p}_{z}+$$
$$q^{5p^2+7p+2}t^{5p^2+3p-1}(1+q^{6p+6}t^{6p+1})(1-q^{2p+2}t^{2p})\binom{n-2p}{p}_{z}-$$
$$q^{5p^2+9p+4}t^{5p^2+5p}(1+q^{2p+2} t^{2p+1})(1+\chi^+_p q^{4p+2}t^{4p-1})\binom{n-2p-1}{p}_{z}],$$
where \(\chi_p^+=0\) when \(p=0\), \(\chi_p^+=1\) for \(p>0\), 
$$z=q^2t^2,\quad\text{and} \quad  \binom{a}{b}_{z}=\frac{(1-z)\cdots (1-z^a)}{ (1-z)\cdots (1-z^{b})(1-z)\cdots (1-z^{a-b})}.$$
\end{conjecture}

The following conjecture is due to  Shumakovitch and  Turner:  

\begin{conjecture}(\cite{shucomm})
\label{shuc}
Let $\mathcal{K}_n(q,t)$ denote the Poincar\'e polynomials of the Khovanov homology of the $(n,n+1)$ torus knot.
Then
\begin{equation}
\label{recursion}
\kK_n(q,t)=\kK_{n-1}(q,t)+\kK_{n-2}(q,t)q^{2n}t^{2n-2}+\kK_{n-3}(q,t)q^{2n+4}t^{2n-1}.
\end{equation}
\end{conjecture}

We prove the following 

\begin{theorem}
If $K_n(q,t)$ is given by the recursion relation (\ref{recursion}) with the appropriate initial conditions and \(P_n(q,t)\) is the Hilbert series of the algebra described in Conjecture~\ref{conj:presentation}, then 
$$\lim_{n\to \infty}K_{n}(q,t)=\lim_{n\to \infty}P_{n}(q,t).$$
\end{theorem}

Finally, we describe an intriguing connection to  the physical models of coloured homology proposed by Gukov,  Walcher and  Sto\v si\'c \cite{gw,gs}. In these models the homology of the unknot is constructed as the Milnor algebra of the certain potential $W_{phys}$ with an isolated singularity. 

\begin{theorem}
The homology of $d_2$ is isomorphic to the Hochschild homology of the category of matrix factorizations of a certain potential $W$.
The potential $W$ has a non-isolated singularity (for $n>1$) and coincided with a bihomogeneous part of $W_{phys}$ of bidegree $(2n+4,2n-2)$.
\end{theorem}

\bigskip

We are grateful to B. Feigin, S. Gukov, M. Hagencamp, M. Khovanov,  A. Kirillov Jr., S. Loktev, L. Rozansky, M. Sto\v si\'c, J. Sussan, O. Viro, and V. Shende for the useful discussions. Special thanks to A. Shumakovitch for providing us with the valuable Khovanov homology data and explaining the Conjecture \ref{shuc}. Most of the computations of the Koszul homology were done using {\tt Singular}, a computer algebra system (\cite{singular}). The research of E. G. was partially supported by the grants
RFBR-10-01-00678, NSh-8462.2010.1 and the Simons foundation.

\section{Evidence for the Conjecture}
\label{sec:evidence}

In this section, we outline the evidence in support of Conjecture~\ref{conj:main}. We verify that the conjecture holds for \(T(n,\infty)\) in the cases \(n=2,3\), where the Khovanov homology is well-understood. We then discuss the computational evidence for larger values of \(n\). 

We define the stable Khovanov homology by
$$ \Kh(T(n,\infty)) := \lim_{m \to \infty} q^{-(n-1)(m-1)+1} \Kh(T(n,m)).$$
It is a theorem of Sto\v si\'c \cite{stosic} that this limit exists. The stable homology is normalized so that 
its Poincar{\'e} polynomial is a polynomial in \(q\) and \(t\) (rather than just a Laurent polynomial), with constant term 1. 
\subsection{\(T(2,\infty)\)}



The Khovanov homology of  \(T(2,\infty)\) is well-known. In  the language of \cite{bnat3}, it can be viewed as dual to the homology of the following chain complex:  
\vskip0.15in
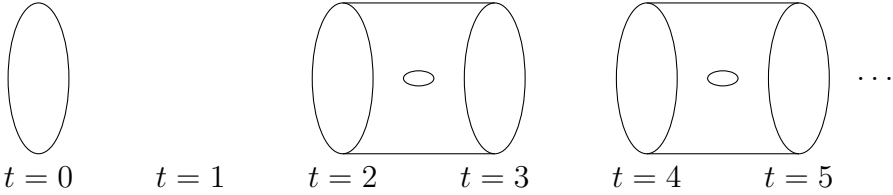
\begin{tikzpicture}
\draw (0,0) ellipse (0.4 and 1);
\draw (4,0) ellipse (0.4 and 1);
\draw (4,1) -- (6,1);
\draw (4,-1) -- (6,-1);
\draw (5,0) ellipse (0.2 and 0.1);
\draw (6,0) ellipse (0.4 and 1);
\draw (8,0) ellipse (0.4 and 1);
\draw (8,1) -- (10,1);
\draw (8,-1) -- (10,-1);
\draw (9,0) ellipse (0.2 and 0.1);
\draw (10,0) ellipse (0.4 and 1);
\draw (11,0) node {$\ldots$};
\draw (0,-1.3) node {$t=0$};
\draw (2,-1.3) node {$t=1$};
\draw (4,-1.3) node {$t=2$};
\draw (6,-1.3) node {$t=3$};
\draw (8,-1.3) node {$t=4$};
\draw (10,-1.3) node {$t=5$};
 \end{tikzpicture}

This picture has the following meaning. The Khovanov homology of the unknot is two-dimensional; as an algebra it can be described as 
$H_0=\mathbb{C}[x_0]/(x_0^2).$ This algebra carries a comultiplication $\mu$ defined by the equations:
$$\mu: H_0\rightarrow H_0\otimes H_0,\quad \mu(1)=1\otimes x_0+x_0\otimes 1,\quad \mu(x_0)=x_0\otimes x_0.$$
Recall that the $q$-degree of $x_0$ is equal to 2.

The complex is generated by an infinite number of copies of $H_0$ in $t$-degrees $0,2,3,4,\ldots$.
The $q$-grading in the $k$-th copy is shifted by $2k$. The maps between the $2k+1$-st copy and the $2k$-th are given by the
cobordism on the picture, which can be presented as a composition of the comultiplication and multiplication:
\begin{equation}
\label{ck2}
H_0\longleftarrow 0\longleftarrow H_0[2]\{4\}\stackrel{m\circ \mu}{\longleftarrow}H_0[3]\{6\}\stackrel{0}{\longleftarrow}H_0[4]\{8\}\stackrel{m\circ \mu}{\longleftarrow}H_0[5]\{10\}\stackrel{0}{\longleftarrow}\cdots
\end{equation}
We remark that $m\circ \mu$ coincides with  multiplication by $2x_0$ and  introduce two formal variables 
$x_1$ of bidegree $q^4t^2$ and $\xi_1$ of bidegree $q^6t^3$, where $\xi_1$ is odd.
In other words, we identify $H_0[2k]\{4k\}$ with $x_1^k\cdot H_0$ and  $H_0[2k+3]\{4k+6\}$ with $x_1^k\xi_1\cdot H_0$.

Then the complex (\ref{ck2}) can be rewritten as an algebra $H_0[x_1,\xi_1]$ with the differential $d(\xi_1)=2x_0x_1$,
which is equivalent to our Koszul model. 

\subsection{\(T(3,\infty)\)}
With rational coefficients, the Khovanov homology of  \(T(3,n)\) was computed by Turner \cite{turner}. The Poincar{\'e} polynomial of the stable homology is
$$P_{3,\infty} = (1+q^2 + q^4t^2) + q^6t^3 \left( \frac{q^2+t+q^2t+q^4t^2+q^6t^2+q^4t^3}{1-q^6t^4}
\right).
$$
\(\Kh_{alg}(3,\infty;\mathbb{Q})\) is computed in section~\ref{subsec:algex}. Its Poincar{\'e} polynomial is easily seen to agree with the one given above.

\subsection{$\mathbb{Z}_2$ coefficients}

In many cases, the Khovanov homology with $\mathbb{Z}_2$ coefficients is simpler that the homology with rational coefficients. It turns out that the stable answers become especially simple if we work over $\mathbb{Z}_2$.

\begin{theorem}
\(\Kh_{alg}{(n,\infty;\mathbb{Z}/2)}\) has the following Poincar\'e series:
$$P_{n}(q,t;\mathbb{Z}_2)=\prod_{i=0}^{n-1}\frac{(1+q^{2i+4}t^{2i+1})}{ (1-q^{2i+2}t^{2i})}\prod_{i=0}^{\lfloor\frac{n-1}{2}\rfloor}\frac{(1-q^{4i+4}t^{4i})}{ (1+q^{4i+4}t^{4i+1})}.$$
\end{theorem}

\begin{proof}
In characteristic 2 the differential $d_2$ degenerates to the following form:
$$d_2(\xi_{2k})=x_k^{2},\quad d_2(\xi_{2k+1})=0.$$
Therefore for every $0\le k\le \lfloor\frac{n-1}{ 2}\rfloor$ the odd generator $\xi_{2i}$ kills $x_i^{2}$ in the homology.
\end{proof}

We used JavaKh \cite{katlas} to verify that \(\Kh_{alg}(n,\infty;\mathbb{Z}/2)\) agrees with \(\Kh(T(n,m))\) in the stable range (\(q\)-degree \(\leq 2m\)) for \((n,m)=(3,50),(4,49),(5,29)\).

\subsection{$\mathbb{Q}$ coefficients}

If we use rational coefficients, the structure of \(\Kh_{alg}\) is more complicated ({\it c.f.} the conjectures in Section~\ref{sec:algebra} below.) The rational Khovanov homology of torus knots has been extensively computed by Shumakovitch \cite{shucomm}. By comparing with his results, we have verified Conjecture~\ref{conj:main} in the stable range up to \((n,m) = (7,20)\). 

In testing Conjecture~\ref{conj:main}, it is important to check that the predictions it makes about Khovanov homology can be distinguished from the ones we would get if we replaced the differential \(d_2(\xi_k) = \sum x_ix_{k-i}\) with  \(d_2'(\xi_k) = \sum \alpha_{ik}x_ix_{k-i}\) for generic values of \(\alpha_{ik}\). In addition to the information on torsion discussed in this section, we can see evidence of this fact with rational coefficients in the case \(n=7\). As discussed in Remark~\ref{rem:generic} below, for generic \(\alpha_{ik}\)  the homology with respect to \(d_2'\) has smaller dimension than the homology with respect to \(d_2\), and the latter groups agree with the actual Khovanov homology. 

More precisely, Remark~\ref{rem:generic} shows that for our choice of \(\alpha_{ik}\) the homology in bidegree $q^{18}t^{13}$ is one-dimensional, while it vanishes for a generic choice. Theorem 6 in \cite{stosic2} states that \(Kh^{i,*}(T(p,q)) \simeq Kh^{i,*}(T(p,q+1))\) for  $i<p+q-2$. Since $13<7+9-2$,  the coefficient at $t^{13}$ in $\Kh(T(7,\infty))$  coincides with the same coefficient for $\Kh(T(7,9))$.
The Poincar\'e polynomial for the Khovanov homology of the latter knot is presented in the Appendix C, and the term $q^{18}t^{13}$ is present.

\subsection{$\mathbb{Z}_{p}$ torsion}

The odd torsion in Khovanov homology was studied in \cite{asaeda} and \cite{shutorsion} (see also \cite{pstorsion})
for some classes of knots, and \cite{bnat2} shows how complicated the torsion can be on the example of $(7,8)$ knot.
It was suggested that  Khovanov homology can have torsion of arbitrarily large order. The following calculation provides support for this claim, as well as some additional evidence in favor of Conjecture~\ref{conj:main}.

\begin{theorem}
Let $p>3$ be a prime number. Then $\Kh_{alg}(p,\infty)$ has nontrivial $\mathbb{Z}_{p}$-torsion at bidegree
$q^{2p+6}t^{2p}$.
\end{theorem}

\begin{proof}
Consider the element $$m=\sum_{i=1}^{p-1}(3i-p)x_{i}\xi_{p-i}=\sum_{i+j=p,1<i<p-1}(2i-j)x_{i}\xi_{j}.$$
Then $$d_2(m)=\sum_{\substack{i+j+k=p,\\ 1<i<p-1}}(2i-j-k)x_ix_jx_k=\sum_{i+j+k=p}(2i-j-k)x_ix_jx_k-2px_px_0^2-px_0\sum_{j+k=p}x_jx_k.$$
Since the first sum vanishes, we have
$$d_2(m)=-2px_px_0^2-px_0\sum_{j+k=p}x_jx_k\cong 0 ~(\mbox{mod}~ p).$$
Since $\deg m=q^{2p+6}t^{2p+1}$, the dimension of the kernel of 
$d_2:C(2p+6,2p+1)\rightarrow C(2p+6,2p)$ jumps by 1 when we reduce it modulo $p$.
Therefore its cokernel has $\mathbb{Z}_{p}$-torsion.
\end{proof}

We have verified the presence of this torsion in Khovanov homology for \(p=5,7\). 

\section{Algebraic structure}
\label{sec:algebra}

We now consider the rational homology of the chain complex appearing in Conjecture~\ref{conj:main}. We will work with rational coefficients for the remainder of the paper. 

\subsection{Koszul model}

Conjecture~\ref{conj:main} tells us to consider the polynomial ring in  even variables $x_0,x_1,\ldots, x_{n-1}$ 
and an equal number of odd variables $\xi_0,\xi_1,\ldots, \xi_{n-1}$,
bigraded as
$$\deg(x_k)=q^{2k+2}t^{2k},\quad \deg(\xi_i)=q^{2i+4}t^{2i+1}.$$
The differential $d_2$ is given by the equation
\begin{equation}
d_2(\xi_m)=\sum_{k=0}^{m}x_{k}x_{m-k}.
\end{equation}
One can check that this differential preserves the $q$-grading and decreases the $t$-grading by 1.

\begin{remark}
Consider the generating functions $$x(z)=\sum_{k=0}^{n-1}x_{k}z^k,\quad \xi(z)=\sum_{k=0}^{n-1}\xi_{k}z^k.$$
If we work over the ring \(\ZZ[z]/(z^n)\), we can express the differential as
 $$d_2(\xi(z))=x(z)^{2}.$$
\end{remark}

At  the  bottom level  of the Koszul complex,  we get the quotient of the polynomial ring $\mathbb{C}[x_i]$ by the ideal $I_n$ generated by the coefficients of the series $x(z)^2$.
It was remarked by  Feigin and Stoyanovsky  \cite{fs} that in the limit $n\to\infty$, this ideal  corresponds to the integrable representation of 
$\widehat{\Sl(2)}$ at level 1, and the equation $x(t)^2$ is an example of the Lepowsky-Primc equations \cite{lp}.  
The bigraded Hilbert series of $\mathbb{C}[x_0,x_1,\ldots]/I_{\infty}$ was computed in \cite{fs} by two different methods, and the equality of the answers corresponds to the following generalization of the 
Rogers-Ramanujan identity (cf. \cite{andrews}):
\begin{equation}
\label{rr}
H_{q,t}(\mathbb{C}[x_0,x_1,\ldots]/I_{\infty})=\sum_{p=0}^{\infty}\frac{q^{2p^2}t^{2p(p-1)}}{(1-q^2t^2)(1-q^4t^4)\ldots (1-q^{2p}t^{2p})}=
\end{equation}
$$\frac{1}{ \prod_{k=1}^{\infty}(1-q^{2k}t^{2k-2})}\sum_{n=0}^{\infty}(-1)^{n}\prod_{k=1}^{n}\frac{(1 - q^{2 k} t^{2 k - 2})}{(1 - q^{2 k} t^{2 k})}(q^{5 n^2 + n} t^{5 n^2 - 3 n} - 
   q^{(n+1)(5n+4)} t^{5n^2+5n}).$$
A similar problem was independently studied by  Brushek, Mourtada and Schepers in \cite{schepers}, where it appeared in the computation of the Hilbert-Poincar\'e series of the arc space of double point.

Most of the algebraic constructions below can be considered as a straightforward generalization of these results  
to the full Koszul homology. In particular, we conjecture the identity (\ref{krr}) that degenerates to (\ref{rr}) at $a=0$.


\subsection{Examples}
\label{subsec:algex}
\begin{example}
\label{ex2}
Let us compute $\Kh_{alg}(2,\infty)$. We have two even generators $x_0, x_1$ and two odd generators $\xi_0, \xi_1$.
Since $d_2(\xi_0)=x_0^2,d_2(\xi_1)=2x_0x_1$, we have a non-trivial homology generator $\mu_0=2x_1\xi_0-x_0\xi_1$. The homology is spanned by the elements
$$p(x_1)+\alpha x_0+r(x_1)\mu_0,$$
where \(p\) and \(r\) are polynomials in \(x_1\) and \(\alpha \in \QQ\). 
(Remark that $x_0\mu_0=d_2(\xi_0\xi_1)$).
The Poincar\'e series has the form 
$$P_{2,\infty}(q,t)=\frac{1+q^8t^3}{ 1-q^4t^2}+q^2.$$
\end{example}

\begin{example}
To compute  $\Kh_{alg}(3,\infty)$, we add the variables $x_2$ and $\xi_2$ with the differential
$d_2(\xi_2)=2x_0x_2+x_1^2.$ 
\end{example}

\begin{lemma}
\label{mu12}
Let
$$\mu_0=2x_1\xi_0-x_0\xi_1,\quad \mu_1=2x_0\xi_2-x_1\xi_1-4x_2\xi_0.$$
Then
\begin{equation}
\label{rels3}
d_2(\mu_0)=d_2(\mu_1)=0,\quad x_0\mu_0=d_2(\xi_0\xi_1),\quad x_0\mu_1-x_1\mu_0=2d_2(\xi_0\xi_2),
\end{equation}
$$2x_2\mu_0+x_1\mu_1=d_2(\xi_1\xi_2),\quad \mu_1\mu_2=-2d_2(\xi_0\xi_1\xi_2).$$
\end{lemma}
One can check that the homology is generated by $\mu_0$ and $\mu_1$ and (\ref{rels3}) is the complete set of relations
between them (this is a special case of Conjectures \ref{gens} and \ref{rels} below).

\begin{lemma}
The homology of $d_2$ is spanned by the elements of the form
$$p_1(x_2)+x_0p_2(x_2)+x_1p_2(x_2)+\alpha \mu_0+\mu_1(q_1(x_2)+x_0q_2(x_2)+x_1q_2(x_2)).$$
\end{lemma}
\medskip

\begin{proof}
Modulo the image of $d_2$, we can eliminate all monomials containing $$x_0^2, x_0x_1, x_1^2, x_0\mu_0, x_1\mu_0, x_2\mu_0, \mu_1\mu_2.$$ 
After this modification the remaining monomials will be linearly independent in the homology.
\end{proof}

\begin{corollary}
$$P_{3,\infty}(q,t)=\frac{(1+q^{10}t^5)(1+q^2+q^4t^2)}{ 1-q^6t^4}+q^8t^3.$$
\end{corollary}

\subsection{Generators and relations}
\label{Sec:mus}

Let us describe the generators in the homology generalizing $\mu_0$ and $\mu_1$ constructed in the previous section.

\begin{lemma}
Consider the set of indeterminates $\varepsilon_{a,b}$, where \(a+b=r\) is fixed, and \(a,b\geq 0\). 
The system of linear equations 
\begin{equation}
\label{abc}
\varepsilon_{a,b+c}+\varepsilon_{b,a+c}+\varepsilon_{c,a+b}=0
\end{equation}
has a nontrivial solution.
\end{lemma}

Remark that the equations in this system are labelled by triples of integers while the variables are 
labelled by pairs. Therefore the number of equations is asymptotically quadratic in \(r\), while the number of variables is asymptotically linear,
and the system is over-determined.

\begin{proof}
Let $\varepsilon_{a,b}=2a-b.$ Then
$$\varepsilon_{a,b+c}+\varepsilon_{b,a+c}+\varepsilon_{c,a+b}=(2a-b-c)+(2b-a-c)+(2c-a-b)=0.$$
\end{proof}

\begin{example}
\label{n6}
Consider the case $r=6$. The system has 7 equations in 7 variables:
$$2\varepsilon_{0,6}+\varepsilon_{6,0}=0,\quad \varepsilon_{0,6}+\varepsilon_{1,5}+\varepsilon_{5,1}=0,$$
$$\varepsilon_{0,6}+\varepsilon_{2,4}+\varepsilon_{4,2}=0,\quad \varepsilon_{0,6}+2\varepsilon_{3,3}=0,$$
$$2\varepsilon_{1,5}+\varepsilon_{4,2}=0,\quad \varepsilon_{1,5}+\varepsilon_{2,4}+\varepsilon_{3,3}=0, \quad \varepsilon_{2,4}=0.$$
Surprisingly, it is has rank 6 and its solution is $$(\varepsilon_{0,6},\varepsilon_{1,5},\varepsilon_{2,4},\varepsilon_{3,3},\varepsilon_{4,2},\varepsilon_{5,1},\varepsilon_{6,0})=(-6,-3,0,3,6,9,12).$$

\end{example}

We are ready to present some non-trivial classes in stable homology.

\begin{lemma}
Let $$\mu_s=\sum_{k=0}^{s+1}\varepsilon_{k,s+1-k}x_{k}\xi_{s+1-k},$$
where the coefficients $\varepsilon_{a,b}$ are defined by the equation (\ref{abc}).
Then $d_2(\mu_s)=0$.
\end{lemma}

\begin{remark}
Since the coefficients of $d_2$  are quadratic in the  $x$-variables, the elements $\mu_s$
do not belong to the image of $d_2$.
\end{remark}

\begin{proof}
We compute 
\begin{align*}
d_2(\mu_s) & =\sum_{k=1}^{s+2}\varepsilon_{k,s+3-k}x_{k}d_2(\xi_{s+3-k}) = \sum_{k=1}^{s+2}\varepsilon_{k,s+3-k}x_{k}\sum_{j=1}^{k-1}x_{j}x_{s+3-k-j} \\ & =\sum_{i+j+k=s+3}(\varepsilon_{i,j+k}+\varepsilon_{j,i+k}+\varepsilon_{k,i+j}) = 0. 
\end{align*}
\end{proof}

\begin{remark}
\label{rem:generic}
 Example \ref{n6} shows that for a generic choice of the coefficients of $d_2$ the corresponding $7\times 7$ 
matrix would be non-degenerate, and the $d_2$ homology would have {\em smaller} dimension. In particular, for a generic choice of the coefficients $d_2$ would have no homology in bidegree $q^{18}t^{13}$. Indeed, the only monomials in this bidegree are $x_{a}\xi_{b}$ with $a+b=6$.
The differential maps the space they span to the space spanned by monomials of the form
  $x_{a}x_{b}x_{c}$ with $a+b+c=6$ according to the matrix from Example \ref{n6}.
\end{remark}

\begin{conjecture}
\label{gens}
The homology of $d_2$ is generated as an algebra by $\mu_s$ and $x_i$. 
\end{conjecture}

\begin{remark} 
In what follows we will use the following  description of the generators $\mu_s$.
Consider the generating function $\mu(z)=\sum_{s=0}^{\infty}z^{s}\mu_s$. Then
\begin{equation}
\mu(z)=2\dot{x}(z)\xi(z)-x(z)\dot{\xi}(z).
\end{equation}
\end{remark}
Let us describe the generalization of the relations (\ref{mu12}).

\begin{lemma}
The following  relations hold in the homology of $d_2$:
\begin{equation}
\label{voarels}
x(z)^2=0,\quad x(z)\mu(z)=0,\quad \ddot{x}(z)\mu(z)-\dot{x}(z)\dot{\mu}(z)=0,\quad \mu(z)\dot{\mu}(z)=0.
\end{equation}
As before, these relations are to be interpreted as holding modulo \(z^n\). 
\end{lemma}

\begin{proof}
$$x\mu=2x\dot{x}\xi-x^2\dot{\xi}=d_2(\xi\dot{\xi})$$
$$\ddot{x}\mu-\dot{x}\dot{\mu}=\ddot{x}(2\dot{x}\xi-x\dot{\xi})-\dot{x}(2\ddot{x}\xi+2\dot{x}\dot{\xi}-\dot{x}\dot{\xi}-x\ddot{\xi})=$$ $$-x\ddot{x}\dot{\xi}-\dot{x}^2\dot{\xi}+x\dot{x}\ddot{\xi}=-\frac{1}{2}d_2(\dot{\xi}\ddot{\xi}).$$
$$\mu\dot{\mu}=(2\dot{x}\xi-x\dot{\xi})(2\ddot{x}\xi+\dot{x}\dot{\xi}-x\ddot{\xi})=$$
$$2\dot{x}^2\xi\dot{\xi}-2x\dot{x}\xi\ddot{\xi}+2x\ddot{x}\xi\dot{\xi}+x^2\dot{\xi}\ddot{\xi}=d_2(\xi\dot{\xi}\ddot{\xi}).$$
\end{proof}

\begin{conjecture}
\label{rels}
The ideal of relations in the Koszul homology is generated by the coefficients of the relations (\ref{voarels}).
\end{conjecture} 
Using {\tt Singular} \cite{singular}, we have verified that both conjectures hold for \(n\leq7\). 
Some further evidence for these conjectures is provided by Theorem~\ref{shu} in the next section. 
 
\begin{remark}
The relations (\ref{voarels}) are not independent, and there are lots of syzygies between them.
For example, first equation presents $d_2(\xi\dot{\xi})$ in terms of $\mu's$, so its derivative presents
$d_2(\xi\ddot{\xi})$ in terms of $\mu$'s. On the other hand, the second equation presents $d_2(\dot{\xi}\ddot{\xi})$ in terms of $\mu$'s.
This suggests a syzygy 
$$d_2(\xi\dot{\xi})d_2(\ddot{\xi})-d_2(\xi\ddot{\xi})d_2(\dot{\xi})+d_2(\dot{\xi}\ddot{\xi})d_2(\xi)=d_2^2(\xi\dot{\xi}\ddot{\xi})=0.$$
\end{remark} 

\begin{lemma}
Assuming Conjecture \ref{gens},   $\Kh_{alg} (n,\infty)$
 contains at most $\lfloor \frac{n+1}{3}\rfloor$ ``levels'',
i. e. the maximal $\xi$--degree of a homology generator is at most  $\lfloor \frac{n+1}{3}\rfloor$.
\end{lemma}

\begin{proof}
Consider the equation $\mu(z)\dot{\mu}(z)=0$. The coefficients of odd powers of \(z\) look like 
$\mu_{i}\mu_{i+1}+\ldots=0$, while the  coefficients of  even powers of $z$ look like $\mu_{i}\mu_{i+2}+\ldots=0$.
(Recall that $\mu_i$ is odd, so there are no terms $\mu_i\mu_i$).
Therefore one can eliminate all monomials containing $\mu_{i}\mu_{i+1}$ and $\mu_{i}\mu_{i+2}$, and the monomial of the maximal $\xi$-degree is $\mu_0\mu_3\cdots \mu_{\left\lfloor\frac{ n-2}{3}\right\rfloor}.$ Its degree is $\left\lfloor\frac{ n-2}{ 3}\right\rfloor+1=\left\lfloor\frac{ n+1}{ 3}\right\rfloor.$
\end{proof}

Finally, we explain some corollaries of  the recent work of 
Feigin \cite{feigin} which  provide further evidence for Conjectures \ref{gens} and \ref{rels}. Feigin studies properties of the ideal $J$ inside $R=\mathbb{C}[\xi_0,\xi_1,\dots,x_0,x_1,\dots]$
generated by the coefficients of the power series:
$$\mu(z),\quad \xi(z)\dot{\xi}(z),\quad \xi(z)\ddot{\xi}(z),\quad \xi(z)\dddot{\xi}(z).$$
An easy computation shows that $J$ is preserved by $d_2$; the main object of study 
of \cite{feigin} is the differential graded
algebra $R_{(1)}=R/J$. Feigin shows that $R_{(1)}$ is a representation of the Virasoro algebra generated by 
$L_i$, $i\in \ZZ$  and $c$:
$$[L_i,c]=0,\quad [L_m,L_n]=(m-n)L_{m+n}+\delta_{m+n}\frac{m^3-m}{3}c.$$
The central element $c$ acts by the constant $-4/5$ on $R_{(1)}$. The algebra $R_{(1)}$ has a natural grading by 
the odd variables:
$$\deg(\xi_i)=1,\quad \deg(x_i)=0.$$ 
The graded components $R_{(1)}[j]$ are subrepresentations and Feigin identifies them with some particular 
highest weight modules of $L$.

 Let us briefly recall the basics of the highest weight theory for the Virasoro 
algebra. 
The algebra $L$ naturally splits into three parts: the positive part $L^+$ generated by $L_i$,  $i>0$;
the negative part $L^-$ generated by $L_i$, $i<0$; and the span of $L_0,c$. From the relations for $L$ we see
that $L_0$ is a grading operator; it is customary to call the eigenspaces of $L_0$ levels. A vector 
in an  $L$-module is called singular if it is anihilated by $L^+$. The Verma module $M_\lambda$ is an $L_-$-module 
freely generated by the singular vector on  the level $\lambda$.

\begin{theorem}\cite{feigin} We have
\begin{enumerate}
 \item The differential $d_2$ commutes with the action of the Virasoro algebra.
 \item The graded component $R_{(1)}[j]$ is isomorphic to the quotient of the Verma 
 module $M_{\lambda_i}$, $\lambda_i=(5i^2-3i)/2$ by a singular vector at level $2i+1$.
 \item For any $i>0$, $H^i(R_{(1)},d_2)=0.$
 \item $H^0(R_{(1)},d_2)$ is the irreducible quotient of $M_0$.
\end{enumerate}
\end{theorem}

\begin{corollary}
The lower level of $\Kh_{alg}(\infty,\infty)$ is the irreducible representation of the Virasoro algebra. 
\end{corollary}
Apart from the  \(\mu_i\), all other generators of the  ideal $J$    have odd degree at least $2$, hence
\begin{corollary}
The first homology of $d_2$ is generated by $\mu_i$ and $x_i$. 
\end{corollary}

\begin{remark}
The theorem above is a particular case of more general result from \cite{feigin} that might be relevant
for studies of $\Sl(N)$ homology. Also it is very plausable that one can extend the result of the last corollary
to  higher homological degrees by some bootstrapping procedure. We hope to return to this question in our future work.
\end{remark}

\subsection{Lee's spectral sequence}


It was conjectured in \cite{dgr} that  Lee's spectral sequence (\cite{lee,slice}) is induced by a differential $d_1$ that commutes with $d_2$. We propose a formula for this differential:
$$d_1(\xi_i)=x_i.$$ 
Remark that if $d_1$ satisfies the Leibnitz rule, it is uniquely defined by the grading restrictions.
Let us consider the spectral sequence induced by $d_1$ on $\Kh_{alg}$.

\begin{example}
Consider $\Kh(T(2,\infty))$. As  was shown in Example \ref{ex2},  
the homology is generated by $x_0, x_1$ and $\mu_0$ modulo the  relations $$x_0^2=2x_0x_1=x_0\mu_0=0.$$
Remark that
$$d_1(\mu_0)=d_1(2x_1\xi_0-x_0\xi_1)=2x_1x_0-x_0x_1=x_0x_1=\frac{1}{2}d_2(\xi_1).$$
This means that the second differential in the spectral sequence (Bar-Natan's knight move) acts as
$$\delta(\mu_0)=d_1\circ d_2^{-1}\circ d_1(\mu_0)=d_1\left(\frac{1}{2}\xi_1\right)=\frac{1}{2}x_1.$$
Therefore $\mu_0$ kills $\xi_1$ by the knight move, and the spectral sequence converges at the $E_3$ page to
the two-dimensional space
$$E_{3}=E_{\infty}=\langle 1,x_0\rangle .$$
\end{example}

\begin{remark}
One can find an apparent contradiction in this result --- the homology of $d_1$ is clearly one-dimensional, while the spectral sequence converges to a two-dimensional space. This problem is caused by the fact that the  homology is infinite dimensional. One can check that for a finite $(2,m)$ knot (i.e. for a suitable finite-dimensional quotient of this complex) the homology of $d_1$ will be two-dimensional --- one generator will be 1, while the degree of the second one will grow as $m$ increases. 
\end{remark}
 
Motivated by this example, we formulate the following algebraic counterpart of the conjectures from \cite{bnat2} and \cite{slice}. 
The following theorem holds modulo Conjectures \ref{gens} and \ref{rels}.
 
\begin{theorem}
\label{e3}
Consider the spectral sequence induced by $d_1$  acting on $\Kh_{alg}(n,\infty)$. 
Then $E_1=E_2=H^{*}(d_2).$ and   $E_3=E_{\infty}=\langle 1,x_0\rangle .$
In particular, the spectral sequence converges at the $E_3$ page.
\end{theorem}

\begin{proof}
Assuming the conjectures, the stable homology is generated by $\mu_s$ and $x_i$, so the multiplicativity of the spectral sequence
allows us to focus on these generators. Remark that
\begin{equation}
\label{d1mu}
d_1(\mu_s)=\sum_{k=0}^{s+1}\varepsilon_{k,s+1-k}x_{k}x_{s+1-k}=\frac{s+1}{2}\sum_{k=0}^{s+1}x_{k}x_{s+1-k}=\frac{s+1}{2}d_2(\xi_{s+1}).
\end{equation}
Here we used the equation
$$\varepsilon_{k,s+1-k}+\varepsilon_{s+1-k,k}=2k-(s+1-k)+2(s+1-k)-k=s+1.$$
We can compute the second differential in the spectral sequence $\delta(\mu_s)$ using the equation (\ref{d1mu}):
$$\delta(\mu_s)=d_1\circ d_2^{-1}\circ d_1(\mu_s)=d_1\left(\frac{s+1}{2}\xi_{s+1}\right)=\frac{s+1}{2}x_{s+1}.$$
Therefore each  even generators $x_i$ is killed by $\mu_{i-1}$.
\end{proof}

\begin{remark}
One can reformulate this proof in terms of the generating series. We have
$$d_1(\mu(z))=d_1(-x(z)\dot{\xi}(z)+2\dot{x}(z)\xi(z))=-x(z)\dot{x}(z)+2x(z)\dot{x}(z)=x(z)\dot{x}(z)=\frac{1}{ 2}d_2(\dot{\xi}(z)).$$
Therefore
$$\delta(\mu(z))=d_1\circ d_2^{-1}\circ d_1(\mu(z))=d_1\left(\frac{1}{2}\dot{\xi}(z)\right)=\frac{1}{2}\dot{x}(z).$$
\end{remark}
One can prove a similar theorem for the reduced homology (see section \ref{secred} below). In the reduced case, the $E_{\infty}$ term will be one-dimensional and spanned by 1.

\section{ Poincar{\'e} polynomials}

\subsection{Bosonic formula}

In this section we give a conjectural formula for the Poincare polynomial of \(\Kh_{alg}(n,\infty) \) for all \(n\). 
This formula comes from  computer experiments, and it can be considered as a potential generalization of the ``bosonic" side of the
Rogers-Ramanujan identity in \cite{fs} (see also \cite{lf,kac,loktev}). It is worth to note that this ``bosonic formula" was obtained in \cite{fs}
using localization  on the affine flag variety for $\widehat{\Sl(2)}$.
We plan to compare this approach with the one proposed below in the future.

 Recall that the lower level of  $\Kh_{alg}(n,\infty)$ can be described by the quotient
of the algebra $\mathbb{C}[x_0,\ldots,x_{n-1}]$ by the ideal generated by the first $n$ coefficients of $x(z)^2$.

\begin{conjecture}
\label{conj:l}
Let  $z=q^2t^2$.  The unreduced Hilbert series for the lower level of $\Kh_{alg}(n,\infty)$ 
 has the  form
\begin{equation}
\label{ln}
L_{n}(q,t)=\frac{1}{\prod_{k=1}^{n}(1-q^{2k}t^{2k-2})}\sum_{p=0}^{\infty}(-1)^{p}\prod_{k=1}^{p}(1-q^{2k}t^{2k-2})\times \\
\end{equation}
$$\left(q^{5p^2+p}t^{5p^2-3p}\binom{n-2p+1}{p}_z -q^{(p+1)(5p+4)}t^{5p^2+5p}\binom{n-2p-1}{p}_z\right).
$$
Here we use the standard $z$-binomial notation: 
$$[m!]_{z}=\prod_{k=1}^{m}\frac{1-z^k}{1-z},\quad \binom{m}{l}_{z}=\frac{[m!]_{z}}{ [l!]_{z}[m-l!]_{z}} \quad (m\ge l).$$
\end{conjecture}

\begin{remark}
In the limit $n\to \infty$ the $z$-binomial coefficients degenerate to simple products:
$$\binom{n-2p+1}{p}_z, \binom{n-2p-1}{p}_z\xrightarrow{n\to\infty} \frac{1}{ (1-z)(1-z^2)\ldots (1-z^p)}=\prod_{k=1}^{p}\frac{1}{ (1-q^{2k}t^{2k})},$$
therefore the equation (\ref{ln}) has a limit 
$$L_{\infty}(q,t)=\frac{1}{ \prod_{k=1}^{n}(1-q^{2k}t^{2k-2})}\sum_{p=0}^{\infty}(-1)^{p}\prod_{k=1}^{p}\frac{(1-q^{2k}t^{2k-2})}{ (1-q^{2k}t^{2k})}\times $$
$$\times \left(q^{5p^2+p}t^{5p^2-3p}-q^{(p+1)(5p+4)}t^{5p^2+5p}\right).$$
This is the right hand side of the generalized Rogers-Ramanujan identity (\ref{rr}), and therefore in this limit 
 Conjecture \ref{conj:l} follows from the results of \cite{fs}.
\end{remark}



One can try to  extend the equation (\ref{ln}) to higher levels of the Koszul homology.

\begin{conjecture}
The unreduced Hilbert series for \(\Kh_{alg}(n,\infty)\) has the form
\begin{equation}
\label{pn}
P_{n}(q,t)=\frac{1}{ \prod_{k=1}^{n}(1-q^{2k}t^{2k-2})}\sum_{p=0}^{\infty}(-1)^{p}\prod_{k=1}^{p}(1-q^{2k}t^{2k-2})\times \\
\end{equation}
$$\times\prod_{k=3p+1}^{n-1}(1+q^{2k+6}t^{2k+1})\prod_{k=1}^{2p-1}(1+q^{2k+2}t^{2k-1})\times$$
$$[q^{5p^2+p}t^{5p^2-3p}(1+\chi_p^+q^{6p+4}t^{6p-1})(1+q^{6p+6}t^{6p+1})\binom{n-2p+1}{p}_{z}+$$
$$q^{5p^2+7p+2}t^{5p^2+3p-1}(1+q^{6p+6}t^{6p+1})(1-q^{2p+2}t^{2p})\binom{n-2p}{p}_{z}-$$
$$q^{5p^2+9p+4}t^{5p^2+5p}(1+q^{2p+2} t^{2p+1})(1+\chi_p^+ q^{4p+2}t^{4p-1})\binom{n-2p-1}{p}_{z}],$$
where \(\chi_p^+ = 0\) when \(p=0\), \(\chi_p^+=1\) for \(p>0\), and the second product inside the  sum
is $1$ when $3p+1>n-1$.
\end{conjecture}

\subsection{Fermionic formula for \(T(\infty, \infty)\)}
Let $\kK_n(q,t)$ denote the Poincar\'e polynomial of \(\Kh(T(n,n+1))\). 
Based on  experimental data,  Shumakovich and Turner conjectured  
that \(\kK_n\) 
satisfies the following recurrence relation.

\begin{conjecture}(\cite{shucomm})

\begin{equation}
\label{STconj}
\kK_n(q,t)=\kK_{n-1}(q,t)+\kK_{n-2}(q,t)q^{2n}t^{2n-2}+\kK_{n-3}(q,t)q^{2n+4}t^{2n-1}.\end{equation}
\end{conjecture}

We construct a combinatorial model for this recursion relation. 
Consider length \(n\)  sequences of 0's and 1's with no blocks of the form 1111 anywhere
and no blocks of the form 111 except possibly at the beginning. Such sequences are split (outside the beginning)
into 1's and 11's separated by blocks of 0's. 
\begin{example}
For $n=3$ all 8 sequences are admissible. For $n=4$ we have 14 sequences: $1111$ and $0111$ are forbidden.
\end{example}

 We weight such sequences by a product of terms corresponding to blocks of \(1\)'s appearing in the sequence. The weights are as follows: 

1) 111 in the beginning: $q^{12}t^{5}$;

2) 1 at position $n$ (first digit has position 0): $q^{2n+2}t^{2n}$;

3) 11 starting at position $n$: $q^{2n+8}t^{2n+3}$.

Let \(K_n\) be the  weighted state sum for length $n$ sequences; that is, the sum of the weights for all such sequences.

\begin{lemma}
\(K_n\) satisfies the recursion relation~\eqref{STconj} and agrees with \(\kK_n\) for \(n=1,2,3\).
\end{lemma}

\begin{proof}
Let us check the recursion relation. The set of length \(n\) sequences ending with 0 contribute  $K_{n-1}$ to \(K_n\).
The sequences ending  with 01 contribute $K_{n-2}(q,t)q^{2n}t^{2n-2}$, and 
the  sequences ending  with 011 contribute $K_{n-3}(q,t)q^{2n+4}t^{2n-1}$.
 The values of $K_n$ for $n=1,2,3$ are easily checked.
\end{proof}

Let us write the formula for the limit $K(q,t)=\lim_{n\to\infty}K_n(q,t)$.

\begin{theorem}
\begin{equation}
\label{fermion}
K(q,t)=\sum_{p=0}^{\infty}q^{2p^2}t^{2p(p-1)}(1+q^{8p+12}t^{8p+5})\frac{(1+q^6t^3)(1+q^8t^5)\ldots (1+q^{2p+4}t^{2p+1})}{ (1-q^2t^2)(1-q^4t^4)\ldots (1-q^{2p}t^{2p})}.
\end{equation}
\end{theorem}

\begin{proof}
Let $U_p(q,t)$ be the state sum giving by summing over all sequences  with $p$ blocks of units, none of which are of length \(3\). 
A sequence with $p+1$ blocks can be one of the following:

1) Starting with 10  at position $k$. This contributes  $q^{2k+2}t^{2k}\cdot (q^2t^2)^{p(k+2)}U_{p}$ to \(U_{p+1}\). 
If we sum over all $k$, we get $$\frac{q^{4p+2}t^{4p}}{ 1-q^{2(p+1)}t^{2(p+1)}}U_p.$$

2) Starting with 110  at position $k$. This contributes $q^{2k+8}t^{2k+3}\cdot (q^2t^2)^{p(k+3)}U_{p}$ to \(U_{p+1}\). 
If we sum over all $k$, we get $$\frac{q^{6p+8}t^{6p+3}}{ 1-q^{2(p+1)}t^{2(p+1)}}U_p.$$
Thus  $$U_{p+1} = \frac{q^{4p+2}t^{4p}(1+q^{2p+6}t^{2p+3})}{ 1-q^{2(p+1)}t^{2(p+1)}}U_p,$$
from which we deduce that 
$$U_{p}(q,t)=q^{2p^2}t^{2p(p-1)}\frac{(1+q^6t^3)(1+q^8t^5)\ldots (1+q^{2p+4}t^{2p+1})}{ (1-q^2t^2)(1-q^4t^4)\ldots (1-q^{2p}t^{2p})}.$$
Let $V_p(q,t)$ denote the state sum where we allow sequences beginning with 111. 
Then $$V_p=U_p+q^{12}t^{5}(q^2t^2)^{4p}U_{p}=(1+q^{8p+12}t^{8p+5})U_p.$$
\end{proof}

We verify directly that the Euler characteristic of \(K(q,t)\) agrees with the stable Jones polynomial of \(T(\infty,\infty)\):

\begin{lemma}
$$K(q,-1)=\frac{1}{ 1-q^2}.$$
\end{lemma}

\begin{proof}
Remark that 
$$U_p(q,-1)=q^{2p^2}\frac{(1-q^6)(1-q^8)\ldots (1-q^{2p+4})}{ (1-q^2)(1-q^4)\ldots (1-q^{2p})}=q^{2p^2}\frac{(1-q^{2p+2})(1-q^{2p+4})}{ (1-q^2)(1-q^4)},$$
$$V_p(q,-1)=(1-q^{8p+12})U_p(q,-1).$$
Therefore we have to prove that
\begin{equation}
\sum_{p=0}^{\infty}q^{2p^2}(1-q^{2p+2})(1-q^{2p+4})(1-q^{8p+12})=1-q^4.
\end{equation}
This follows from the direct expansion of the left hand side: all terms will cancel out
except $1$ and $-q^4$.
\end{proof}

Comparing  the ``fermionic" formula \eqref{fermion}) with the ``bosonic'' formula \eqref{pn} in the limit \(n\to \infty\) suggests the following identity.
The $a$-grading from the HOMFLY homology can be traced on both sides.

\begin{conjecture}({\it ``Khovanov-Rogers-Ramanujan identity''})
Let
$$A(a,q,t)=\sum_{p=0}^{\infty}q^{2 p^2}t^{2 p (p - 1)}(1 + a^2 q^{8 p + 8}t^{8 p + 5})
\prod_{j=1}^{p}\frac{(1 + a^2 q^{2 j}t^{2 j + 1})}{ (1 - q^{2 j}t^{2 j})}.$$
and
$$B(a,q,t)=\frac{1}{ \prod_{k=1}^{\infty}(1-q^{2k}t^{2k-2})}\sum_{p=0}^{\infty}(-1)^{p}\prod_{k=1}^{p}\frac{(1-q^{2k}t^{2k-2})}{(1-q^{2k}t^{2k})}\times $$
$$\times\prod_{k=3p+1}^{\infty}(1 + a^2 q^{2 k + 2} t^{2 k + 1})\prod_{k=1}^{2p-1}(1 + a^2 q^{2 k - 2} t^{2 k - 1})\times$$
$$[q^{5p^2+p}t^{5p^2-3p}(1+\chi_p^+ a^2q^{6p}t^{6p-1})(1+a^{2}q^{6p + 2}t^{6p+1})+$$
$$a^{2}q^{5p^2+7p-2}t^{5p^2+3p-1}(1+a^2q^{6p+2}t^{6p+1})(1-q^{2p+2}t^{2p})-$$
$$q^{5p^2+9p+4}t^{5p^2+5p}(1+a^2q^{2p-2} t^{2p+1})(1+\chi_p^+ a^{2}q^{4p-2}t^{4p-1})].$$
Then 
\begin{equation}
\label{krr}
A(a,q,t)=B(a,q,t).
\end{equation}
\end{conjecture}
Using a computer, we have checked that this identity holds up through terms of order \(q^{100}\). 

The following theorem provides some evidence in support of Conjectures \ref{gens} and \ref{rels}.

\begin{theorem}
\label{shu}
 The Hilbert series of the algebra generated by \(\mu_n\) (\(n\geq 0\)) and \(x_i\) (\(i \geq 0\)) and satisfying the relations in equation  \eqref{voarels} is $K(q,t)$.
\end{theorem}
 
\begin{proof}
The elements \(\mu_n\)  have  grading 
$$\deg(\mu_n)=q^{2n+8}t^{2n+3}.$$
Let us return to our combinatorial model:

1) 1 at position $n$ corresponds to $x_n$.

2) 11 starting at position $n$ corresponds to $\mu_n$.

3) 111 in the beginning corresponds to $x_0\mu_1$.

We have to check that we can eliminate the following products using the relations:

$$ x_ix_{i+1},\quad x_i\mu_{i+1}\quad (i>0),\quad  \mu_{i}x_{i+2},\quad \mu_i\mu_{i+2},$$
$$x_i^2,\quad x_i\mu_i,\quad x_{i+1}\mu_i,\quad \mu_{i}\mu_{i+1}.$$
We can eliminate $x_i^2$ and $x_ix_{i+1}$ using the equation $x(z)^2=0$;
$\mu_{i}\mu_{i+1}$ and $\mu_i\mu_{i+2}$ using the equation  $\mu(t)\dot{\mu}(z)=0$.
Finally, we can eliminate $x_{i-1}\mu_{i+1}, x_{i}\mu_i, x_{i+1}\mu_i, x_{i+2}\mu_{i}$
using two remaining equations $$x(z)\mu(z)=\ddot{x}(z)\mu(z)-\dot{x}(z)\dot{\mu}(z)=0.$$
\end{proof}

\section{Reduced homology}
\label{secred}

In this section we briefly review the structure of the reduced stable homology. The computations in \cite{GORS} suggest that the construction and the differential should be similar to the  unreduced case, except that  $x_0$ and $\xi_0$ are omitted. To be  specific, 
consider the polynomial ring in  even variables $x_1,\ldots, x_{n-1}$ 
and equal number of odd variables $\xi_1,\ldots, \xi_{n-1}$ 
bigraded as
$$\deg(x_k)=q^{2k+2}t^{2k},\quad \deg(\xi_i)=q^{2i+4}t^{2i+1}.$$
The differential $d_2$ is given by the equation
\begin{equation}
d_2(\xi_m)=\sum_{k=1}^{m-1}x_{k}x_{m-k}.
\end{equation}

\begin{conjecture}
The stable reduced Khovanov homology of $T(n,\infty)$ is isomorphic to the homology of 
$\mathbb{C}[x_1\ldots,x_{n-1}]\otimes \Lambda^*[\xi_1,\ldots,\xi_{n-1}]$ with respect to $d_2$.
\end{conjecture}

\begin{example} (cf. \cite{g})
Let us compute $\Kh_{alg}^{red}(4,\infty)$. 
We have
$$d_2(\xi_1)=0,\quad d_2(\xi_2)=x_1^2,\quad d_2(\xi_3)=2x_1x_2.$$
As in Example \ref{ex2}, we can introduce a homology generator 
$\overline{\mu_0}=2x_2\xi_2-x_1\xi_3$ and check that the homology is spanned 
by the expressions of the form $p(x_2)+\alpha x_1+\overline{\mu_0}q(x_2)$
up to multiplication by polynomials in $\xi_1$ and $x_3$. 
Therefore the Poincar\'e series has the  form
$$\overline{P}_{4,\infty}(q,t)=\frac{1+q^6t^3}{ 1-q^8t^6}\left(q^4t^2+\frac{1+q^8t^5}{ 1-q^6t^4}\right).$$
\end{example}

Remark that the reduced and unreduced differentials look similar up to a shift of grading. 
Modulo multiplication by $\xi_1$ and $x_{n-1}$, we can replace 
$\xi_{i}$ by $\xi_{i-2}$ and $x_i$ by $x_{i-1}$ to get the unreduced stable homology of the $(n-2,\infty)$ knot.
We get the following result.

\begin{lemma}
$ \displaystyle \Kh_{alg}^{red}(n,\infty) \simeq \Kh_{alg}(n-2,\infty) \otimes \CC[x_{n-1},\xi_1]. $
\end{lemma}

Note that this isomorphism does not respect the \(q\) and \(t\) gradings. However, it is not difficult to  reconstruct the grading shifts for this correspondence and obtain an analogue of  equation (\ref{pn}) 
for the reduced homology.  

\begin{conjecture}
The Poincar{\'e} series of \(\Kh_{alg}^{red}(n,\infty)\) has the form
\begin{equation}
\label{pnred}
\overline{P}_{n}(q,t)=\frac{(1+q^6t^3)}{ \prod_{k=1}^{n-1}(1-q^{2k+2}t^{2k})}\sum_{p=0}^{\infty}(-1)^{p}\prod_{k=1}^{p}(1-q^{2k+2}t^{2k})\times \\
\end{equation}
$$\times\prod_{k=3p+1}^{n-1}(1+q^{2k+12}t^{2k+7})\prod_{k=1}^{2p-1}(1+q^{2k+6}t^{2k+3})\times$$
$$[q^{5p^2+5p}t^{5p^2+p}(1+q^{6p+10}t^{6p+5})(1+q^{6p+12}t^{6p+7})\binom{n-2p+1}{p}_{z}+$$
$$q^{5p^2+11p+6}t^{5p^2+7p+3}(1+q^{6p+12}t^{6p+7})(1-q^{2p+4}t^{2p+2})\binom{n-2p}{p}_{z}-$$
$$q^{5p^2+13p+8}t^{5p^2+9p+4}(1+q^{2p+4} t^{2p+3})(1+q^{4p+6}t^{4p+3})\binom{n-2p-1}{p}_{z}],$$
where  the second product under the inside sum
is $1$ when $3p+1>n-1$.
\end{conjecture}


\begin{remark}
It was conjectured in \cite{dgr} that 
$$\overline{P}_{2,\infty}(q,t)=(1+q^6t^3)(1+q^4t^2+q^8t^4+\ldots),$$
$$\overline{P}_{3,\infty}(q,t)=\frac{1+q^4t^2+q^6t^3+q^{10}t^5}{1-q^6t^4},$$
$$\overline{P}_{4,\infty}(q,t)=\frac{(1+q^6t^3)}{ (1-q^8t^6)}\left(1+q^4t^2+\frac{q^6t^4(1+q^8t^5)}{ 1-q^6t^4}\right).$$
One can check that these answers coincide with the above construction. (See Appendix B for a comparison).
\end{remark}

\section{Matrix factorizations}

In different physical models of knot homology (e.g. \cite{gw},\cite{gs}) the colored homology of the unknot
is described in terms of matrix factorizations. Following Rozansky's observation that the infinite torus braid is a categorified Jones-Wenzl projector \cite{rozansky}, we  focus on the 
unknot coloured by the $n$th symmetric power of the fundamental representation of $\Sl(N)$.

\begin{definition}(\cite{eisenbud})
A matrix factorisation of a function $W$ over a ring $R$ is a pair $(M,d)$, where $M=M_0+M_1$ is a $\mathbb{Z}_2$-graded  $R$-module of finite rank equipped with an $R$-linear map $d$ of odd degree  satisfying the equation $d^2=W\cdot id_{M}.$
\end{definition}

We will need the following basic facts about matrix factorizations:

\begin{theorem}
\label{factsmf}
(a) (\cite{hkr}) The Hochschild cohomology of the algebra of functions on $\mathbb{C}^{n}$  is equal to the 
algebra of polyvector fields on $\mathbb{C}^{n}$.

(b) (\cite{dyckerhoff}) Consider a function $W:\mathbb{C}^{n}\rightarrow \mathbb{C}.$ Then Hochschild cohomology $\HHH(\MF(W))$ of the category of
matrix factorizations of $W$ is  equal to  the Koszul homology of the complex obtained from polyvector fields by the contraction with $dW$. 

(c) (e.g. \cite{dyckerhoff}) If $W$ has an isolated singularity, then $\HHH(\MF(W))$  is isomorphic to the Milnor algebra of $W$ at this singularity:
$$\HHH(\MF(W))=\mathbb{C}[x_1,\ldots,x_n]/\left(\frac{\partial W}{\partial x_1},\ldots,\frac{\partial W}{ \partial x_n}\right).$$
\end{theorem}

Remark that we can dualize the Koszul complex (b) and obtain the differential 
$$D_{W}(dx_{i})=\frac{\partial W}{ \partial x_i}$$
acting as a derivation on the algebra $\mathbb{C}[x_1,\ldots,x_n,dx_1,\ldots,dx_n]$ of differential forms on $\mathbb{C}^n$.
Now  part (c) follows from the well known fact that $W$ has an isolated singularity if and only if its partial derivatives form a regular sequence.

The following potentials (with isolated singularities) were proposed for the totally symmetric representations by Gukov and Walcher:
\begin{conjecture}(\cite{gw}) The generating function for the $(\Sl(N),S^{k})$ potentials has the form:
$$\sum_{N=0}^{\infty}z^{N+k}(-1)^{N}W_{\Sl(N),S^{k}}(x_1,\ldots,x_k)=\left(1+\sum_{i=1}^{k}z^{i}x_{i}\right)\ln\left(1+\sum_{i=1}^{k}z^{i}x_{i}\right)$$
The$(\Sl(N),S^{k})$ colored homology of the unknot coincides with the Milnor algebra of the corresponding potential.
\end{conjecture}

We slightly change notations and write

$$W_{phys}(x_0,\ldots,x_{n-1})=W_{\Sl(2),S^{n}}(x_0,\ldots,x_{n-1})=\Coef_{n+2}\left[(1+zx(z))\ln(1+zx(z))\right],$$
where 
$x(z)=\sum_{i=0}^{n-1}z^{i}x_i.$

Let us assume that the $x_i$ are bigraded as above: $\deg(x_i)=q^{2i+2}t^{2i}.$ Then the differential $D_{W}$ will preserve both gradings iff
$W$ is bihomogeneous. Let $\overline{W}_n(x_0,\ldots,x_{n-1})$ be the piece of bidegree $(2n+4,2n-2)$ in $W_{phys}$.

\begin{lemma}
The potential $\overline{W}_n$ can be written as follows:
$$\overline{W_n}(x_0,\ldots,x_{n-1})=-\frac{1}{6}\Coef_{n-1}[x(z)^{3}].$$
\end{lemma}

\begin{proof}
Remark that the difference between $q$- and $t$-gradings for $x_i$ is equal to 2. Therefore the piece of bidegree $(2n+4,2n-2)$ should be cubic in $x_i$. Now
$$(1+zx(z))\ln(1+zx(z))=zx(z)+\frac{1}{2}z^2x(z)^{2}-\frac{1}{6}z^3x(z)^{3}+\ldots,$$
so the cubic part equals to $-\frac{1}{6}z^3x(z)^{3}$.
\end{proof}

\begin{example}
We have $$\overline{W}_0=-\frac{1}{6}x_0^3,\quad \overline{W}_1=-\frac{1}{2}x_0^2x_1.$$
Remark that $\overline{W}_1$ has a non-isolated singularity.
\end{example}

\begin{theorem}
The Hochschild homology of the category of matrix factorizations of the potential $\overline{W}_n$ 
is isomorphic to the homology of $d_2$.
\end{theorem}

\begin{proof}
By Theorem \ref{factsmf}, it is sufficient to study the Koszul complex associated with the partial derivatives of  $\overline{W}_n$.
We have
$$\frac{\partial}{ \partial x_i}\overline{W}_n=-\frac{1}{6}\frac{\partial}{ \partial x_i}\Coef_{n-1}[x(z)^{3}]=
-\frac{1}{2}\Coef_{n-1}\left[x(z)^2\frac{\partial}{ \partial x_i}x(z)\right]=$$ 
$$-\frac{1}{2}\Coef_{n-1}\left[z^{i}x(z)^2\right]=-\frac{1}{2}\Coef_{n-1-i}[x(z)^2].$$
Therefore 
$$D_{\overline{W}_n}(dx_{n-1-i})=\frac{\partial \overline{W}_n}{ \partial x_{n-1-i}}=-\frac{1}{2}\Coef_{i}[x(z)^2]=-\frac{1}{2}\sum_{j=0}^{i}x_jx_{j-i}.$$
It suffices to identify
$$\xi_{i}=-2dx_{n-1-i}.$$
\end{proof}

{\bf \Large Appendix}
\begin{appendix}
\section{Unreduced Poincar\'e series}

Here we collect the answers for the Poincar\'e series of \(\Kh_{alg}(n,\infty) \) for $n\le 7$. These series were computed using {\tt Singular} \cite{singular}. The resulting series coincide with Shumakovitch's computations \cite{shucomm} up to high $q$-degree.
For example, for $(7,20)$ torus knot the first difference is in $q$-degree 42.

For the reader's convenience, we multiply both parts of  equation (\ref{pn}) by $\prod_{i=1}^{n}(1-q^{2i}t^{2i-2}).$

$$(1-q^2)(1-q^4t^2)P_2(q,t)=(1 + q^{8}t^{3})-q^{4}(1 + q^{8}t^3) - q^{6}t^{2}(1 - q^2)(1 + q^{4}t)=$$
$$(1-q^2)(1 + q^2-q^{6}t^2 + q^{8}t^3);$$

$$\prod_{i=1}^{3}(1-q^{2i}t^{2i-2})P_3(q,t)= (1 + q^8t^3)(1 + q^{10}t^5) - 
 q^4(1 + q^8t^3)(1+q^{10}t^5) - $$ $$
 q^6t^2(1 - q^2)(1 + q^4t)(1 + q^2t^2)(1 + q^{10}t^5) - 
 q^{14}t^7(1 - q^2)(1 - q^{4}t^2)(1 + q^4t)=$$
 $$(1-q^2)(1-q^4t^2)(1 + q^2 + q^4 t^2 + q^8 t^3 + q^{10} t^5 + q^{12} t^5);$$

$$\prod_{i=1}^{4}(1-q^{2i}t^{2i-2})P_4(q,t) = (1 + q^{8}t^3)(1 + q^{10}t^5)(1 + q^{12}t^7) - 
 q^4(1 + q^8t^3)(1 + q^{10}t^5) (1 + q^{12}t^7) - $$ $$
 q^6t^2(1 - q^2)(1 + q^4t)(1 + q^{10}t^5)(1 + q^{12}t^7)(1 + q^{2}t^2 + q^{4}t^4) - $$ $$
 q^{14}t^7(1 - q^2)(1 - q^4t^2)(1 + q^4t)(1 + q^{12}t^7)(1 + q^2t^2) + 
 q^{18}t^{10}(1 - q^2)(1 + q^4t)(1 + q^6t^3)(1 + q^4t^3);$$

$$\prod_{i=1}^{5}(1-q^{2i}t^{2i-2})P_5(q,t) = \prod_{i=1}^{4}(1+q^{2i+6}t^{2i+1}) -
 q^4 \prod_{i=1}^{4}(1+q^{2i+6}t^{2i+1})  - $$ $$
  q^6t^2(1 - q^2)(1 + q^4t)(1 + q^2t^2 + q^4t^4 + q^6t^6) \prod_{i=2}^{4}(1+q^{2i+6}t^{2i+1}) - $$ $$
 q^{14}t^7(1 - q^2)(1 - q^4t^2)(1 + q^{4}t)(1 + q^{12}t^7)(1 + q^{14}t^9)(1 + q^2t^2 + q^4t^4) + $$ $$
 q^{18}t^{10}(1 - q^2)(1 + q^4t)(1 + q^{6}t^3)(1 + q^2t^2)(1 + q^{14}t^9)(1 + q^4t^3) + $$ $$
 q^{22}t^{14}(1 - q^2)(1 - q^4t^2)(1 + q^{4}t)(1 + q^6t^3)(1 + q^8t^5);$$

$$\prod_{i=1}^{6}(1-q^{2i}t^{2i-2})P_6(q,t) = \prod_{i=1}^{5}(1+q^{2i+6}t^{2i+1}) -
 q^4 \prod_{i=1}^{5}(1+q^{2i+6}t^{2i+1})  - $$ $$
 q^6t^2(1 - q^2)(1 + q^4t)(1 + q^2t^2 + q^4t^4 + q^6t^6 + q^8t^8)\prod_{i=2}^{5}(1+q^{2i+6}t^{2i+1}) - $$ $$
 q^{14}t^7(1 - q^2)(1 - q^4t^2)(1 + q^4t)(1 + q^{12}t^7)(1 + q^{14}t^9)(1 + q^{16}t^{11})(1 + q^2t^2 + q^4t^4 + q^6t^6) + $$ $$
 q^{18}t^{10}(1 - q^2)(1 + q^4t)(1 + q^6t^3)(1 + q^2t^2 + q^4t^4)(1+ q^{14}t^9)(1 + q^{16}t^{11})(1 + q^4t^3) + $$ $$
 q^{22}t^{14}(1 - q^2)(1 - q^4t^2)(1 + q^4t)(1 + q^6t^3)(1 + q^8t^5)(1 + q^{16}t^{11})(1 + q^2t^2 + q^4t^4) + $$ $$
 q^{36}t^{25}(1 - q^2)(1 - q^4t^2)(1 - q^6t^4)(1 + q^4t)(1 + q^6t^3)(1 + q^8t^5);$$

$$\prod_{i=1}^{7}(1-q^{2i}t^{2i-2})P_7(q,t) = \prod_{i=1}^{6}(1+q^{2i+6}t^{2i+1}) -
 q^4 \prod_{i=1}^{6}(1+q^{2i+6}t^{2i+1})  - $$ $$
 q^6t^2(1 - q^2)(1 + q^4t)(1 + q^2t^2 + q^4t^4 +q^6t^6 + q^8t^8 + q^{10}t^{10})\prod_{i=2}^{6}(1+q^{2i+6}t^{2i+1}) - $$ $$
 q^{14}t^7(1 - q^2)(1 - q^4t^2)(1 + q^4t)(1 + q^2t^2 + q^4t^4 +q^6t^6 + q^8t^8)\prod_{i=3}^{6}(1+q^{2i+6}t^{2i+1}) + $$ $$
 q^{18}t^{10}(1 - q^2(1 + q^4t)(1 + q^6t^3)(1 + q^4t^3)(1 + q^2t^2 + q^4t^4 + q^6t^6) (1 + q^{14}t^9)(1 + q^{16}t^{11})(1 + q^{18}t^{13}) + $$ $$
 q^{22}t^{14}(1 - q^2)(1 - q^4t^2)(1 + q^4t)(1 + q^6t^3)(1 + q^8t^5)\times$$ $$ \times(1 + q^{16}t^{11})(1 + q^{18}t^{13})(1 + q^2t^2 + 2q^4t^4 + q^6t^6 + q^8t^8) +$$ $$
 q^{36}t^{25}(1 - q^2)(1 - q^4t^2)(1 - q^6t^4)(1 + q^4t)(1 + q^6t^3)(1 + q^8t^5)(1 + q^{18}t^{13})(1 + q^2t^2 + q^4t^4) - $$ $$
 q^{42}t^{30}(1 - q^2)(1 - q^4t^2)(1 + q^4t)(1 + q^6t^3)(1 + q^8t^5)(1 + q^{10}t^7)(1 + q^6t^5);$$

\section{Reduced Poincar\'e series}

Here we collect the answers for the conjectural Poincar\'e series of stable 
reduced Khovanov homology for $n\le 7$.  The resulting series coincide with the data from \cite{shucomm} up to high $q$-degree.
For example, for the $(5,49)$ torus knot the first difference is in $q$-degree 100,
for $(6,25)$ the first difference is in $q$-degree 52, and for $(7,15)$ the first difference is in $q$-degree 32.

$$\overline{P}_{3}(q,t)=\frac{1+q^6t^3}{ \prod_{i=1}^{2}(1-q^{2i+2}t^{2i})}[1-q^8t^4];$$
$$\overline{P}_{4}(q,t)=\frac{1+q^6t^3}{ \prod_{i=1}^{3}(1-q^{2i+2}t^{2i})}[(1 + q^{14}t^{9})-q^{8}t^{4}(1 + q^{14}t^9) - q^{10}t^{6}(1 - q^4t^2)(1 + q^{8}t^{5})];$$
$$\overline{P}_{5}(q,t)=\frac{1+q^6t^3}{ \prod_{i=1}^{4}(1-q^{2i+2}t^{2i})}[(1 + q^{14}t^{9})(1 + q^{16}t^{11}) - 
 q^8t^4(1 + q^{14}t^{9})(1+q^{16}t^{11}) - $$ $$
 q^{10}t^6(1 - q^4t^2)(1 + q^8t^5)(1 + q^2t^2)(1 + q^{16}t^{11}) - 
 q^{22}t^{15}(1 - q^4t^2)(1 - q^{6}t^4)(1 + q^8t^5)];$$

$$\overline{P}_{6}(q,t)=\frac{1+q^6t^3}{ \prod_{i=1}^{5}(1-q^{2i+2}t^{2i})}[(1 + q^{14}t^9)(1 + q^{16}t^{11})(1 + q^{18}t^{13}) - $$ $$
 q^8t^4(1 + q^{14}t^9)(1 + q^{16}t^{11})(1 + q^{18}t^{13}) - $$ $$
 q^{10}t^6(1 - q^4t^2)(1 + q^8t^5)(1 + q^{16}t^{11})(1 + q^{18}t^{13})(1 + q^{2}t^2 + q^{4}t^4) - $$ $$
 q^{22}t^{15}(1 - q^4t^2)(1 - q^6t^4)(1 + q^8t^5)(1 + q^{18}t^{13})(1 + q^2t^2) + $$ $$
 q^{26}t^{18}(1 - q^4t^2)(1 + q^8t^5)(1 + q^{10}t^{7})(1 + q^6t^5)];$$

$$\overline{P}_{7}(q,t)=\frac{1+q^6t^3}{ \prod_{i=1}^{6}(1-q^{2i+2}t^{2i})}[\prod_{i=1}^{4}(1+q^{2i+12}t^{2i+7}) -
 q^8t^4 \prod_{i=1}^{4}(1+q^{2i+12}t^{2i+7})  - $$ $$
  q^{10}t^{6}(1 - q^4t^2)(1 + q^8t^5)(1 + q^2t^2 + q^4t^4 + q^6t^6) \prod_{i=2}^{4}(1+q^{2i+12}t^{2i+7}) - $$ $$
 q^{22}t^{15}(1 - q^4t^2)(1 - q^6t^4)(1 + q^{8}t^5)(1 + q^{18}t^{13})(1 + q^{20}t^{15})(1 + q^2t^2 + q^4t^4) + $$ $$
 q^{26}t^{18}(1 - q^4t^2)(1 + q^8t^5)(1 + q^{10}t^7)(1 + q^2t^2)(1 + q^{20}t^{15})(1 + q^6t^5) + $$ $$
 q^{30}t^{22}(1 - q^4t^2)(1 - q^6t^4)(1 + q^{8}t^5)(1 + q^{10}t^7)(1 + q^{12}t^{9})].$$


\section{$(7,9)$ torus knot}

We present the exact normalized Poincar\'e polynomial for the unreduced $\mathbb{Q}$-Khovanov homology of the $(7,9)$ torus knot,
obtained with JavaKh (\cite{katlas}).

$$q^{-47}P(T(7,9))=1 + q^2 + q^4 t^2 + q^8 t^3 + q^6 t^4 + q^8 t^4 + q^{10} t^5 + 
 q^{12} t^5 + q^8 t^6 + q^{10} t^6 + q^{12} t^7 + q^{14} t^7 +$$ $$ q^{10} t^8 + 
2 q^{12} t^8 + q^{14} t^9 + 2 q^{16} t^9 + q^{12} t^{10} + 2 q^{14} t^{10} + 
 q^{16} t^{11} + 3 q^{18} t^{11} + q^{14} t^{12} + 3 q^{16} t^{12} + q^{18} t^{12} + $$
$$ q^{22} t^{12} + q^{18} t^{13} + 4 q^{20} t^{13} + q^{22} t^{13} + 3 q^{18} t^{14} + 
 q^{20} t^{14} + q^{24} t^{14} + 5 q^{22} t^{15} + 2 q^{24} t^{15} + 2 q^{20} t^{16} + $$
$$ 3 q^{22} t^{16} + 2 q^{26} t^{16} + q^{28} t^{16} + 4 q^{24} t^{17} + 4 q^{26} t^{17} + 
 q^{22} t^{18} + 3 q^{24} t^{18} + 2 q^{28} t^{18} + q^{30} t^{18} + 2 q^{26} t^{19} + $$
$$ 5 q^{28} t^{19} + q^{24} t^{20} + 3 q^{26} t^{20} + q^{30} t^{20} + 2 q^{32} t^{20} + 
 q^{28} t^{21} + 5 q^{30} t^{21} + 2 q^{28} t^{22} + q^{30} t^{22} + 2 q^{34} t^{22} + $$
$$ 5 q^{32} t^{23} + q^{34} t^{23} + q^{30} t^{24} + 3 q^{32} t^{24} + 3 q^{36} t^{24} + 
 2 q^{34} t^{25} + 4 q^{36} t^{25} + q^{34} t^{26} + q^{36} t^{26} + 2 q^{38} t^{26} + $$
$$ q^{40} t^{26} + 3 q^{38} t^{27} + q^{40} t^{27} + q^{42} t^{27} + q^{38} t^{28} + 
 2 q^{42} t^{28} + 2 q^{42} t^{29} + q^{42} t^{30} + q^{46} t^{30} + q^{46} t^{31}.$$

The total dimension of the homology is equal to 134.
One can observe the multiplicative generators of the following (bi)degrees:

$$\deg(x_0)=q^2,~\deg(x_1)=q^4t^2,~\deg(x_2)=q^6t^4,~\deg(x_3)=q^8t^6,$$
$$\deg(x_4)=q^{10}t^8,~\deg(x_5)=q^{12}t^{10},~\deg(x_6)=q^{14}t^{12},$$
$$\deg(\mu_0)=q^8t^3,~\deg(\mu_1)=q^{10}t^5,~\deg(\mu_2)=q^{12}t^7,$$
$$\deg(\mu_3)=q^{14}t^9,~\deg(\mu_4)=q^{16}t^{11},~\deg(\mu_5)=q^{18}t^{13}.$$

Let us consider $\mathbb{Z}_2$ coefficients. The Poincar\'e polynomial is equal to:

$$q^{-47}P_{2}(T(7,9))=1 + q^2 + q^4 t^2 + q^6 t^2 + q^6 t^3 + q^8 t^3 + q^6 t^4 + q^8 t^4 + 
 q^{10}t^5 + q^{12}t^5 + q^8 t^6 + 2 q^{10}t^6 +$$ $$ q^{12}t^6 + q^{10}t^7 + 
2 q^{12}t^7 + q^{14}t^7 + q^{10}t^8 + 2 q^{12}t^8 + q^{14}t^8 + 
 2 q^{14}t^9 + 3 q^{16}t^9 + q^{18}t^9 + q^{12}t^{10} + 3 q^{14}t^{10} + $$
$$ 3 q^{16}t^{10} + q^{18}t^{10} + q^{14} t^{11} + 3 q^{16} t^{11} + 3 q^{18} t^{11} + 
 q^{20} t^{11} + q^{14} t^{12} + 3 q^{16} t^{12} + 3 q^{18} t^{12} + 2 q^{20} t^{12} +$$ 
$$ q^{22} t^{12} + 3 q^{18} t^{13} + 6 q^{20} t^{13} + 3 q^{22} t^{13} + 3 q^{18} t^{14} + 
 5 q^{20} t^{14} + 3 q^{22} t^{14} + q^{24} t^{14} + 3 q^{20} t^{15} + 6 q^{22} t^{15} +$$ 
$$ 4 q^{24} t^{15} + q^{26} t^{15} + 2 q^{20} t^{16} + 4 q^{22} t^{16} + 4 q^{24} t^{16} + 
 3 q^{26} t^{16} + q^{28} t^{16} + q^{22} t^{17} + 7 q^{24} t^{17} + 7 q^{26} t^{17} + $$
$$ q^{28} t^{17} + q^{22} t^{18} + 5 q^{24} t^{18} + 6 q^{26} t^{18} + 3 q^{28} t^{18} + 
 q^{30} t^{18} + q^{24} t^{19} + 5 q^{26} t^{19} + 7 q^{28} t^{19} + 3 q^{30} t^{19} + $$
$$ q^{24} t^{20} + 3 q^{26} t^{20} + 5 q^{28} t^{20} + 6 q^{30} t^{20} + 3 q^{32} t^{20} + 
 4 q^{28} t^{21} + 9 q^{30} t^{21} + 5 q^{32} t^{21} + 2 q^{28} t^{22} + $$
$$ 5 q^{30} t^{22} + 5 q^{32} t^{22} + 3 q^{34} t^{22} + q^{36} t^{22} + 2 q^{30} t^{23} + 
 6 q^{32} t^{23} + 6 q^{34}t^{23} + 2 q^{36} t^{23} + q^{30} t^{24} + 3 q^{32} t^{24} + $$
$$ 6 q^{34} t^{24} + 5 q^{36} t^{24} + q^{38} t^{24} + 4 q^{34} t^{25} + 6 q^{36} t^{25} + 
 2 q^{38} t^{25} + q^{34} t^{26} + 3 q^{36} t^{26} + 4 q^{38} t^{26} + 2 q^{40} t^{26} + $$
$$ q^{36} t^{27} + 4 q^{38} t^{27} + 4 q^{40} t^{27} + q^{42} t^{27} + q^{38} t^{28} + 
 3 q^{40} t^{28} + 2 q^{42} t^{28} + q^{40} t^{29} + 2 q^{42} t^{29} + q^{44} t^{29} + $$
$$ q^{42} t^{30} + 2 q^{44} t^{30} + q^{46} t^{30} + q^{44} t^{31} + q^{46} t^{31}.$$

\noindent The total dimension of the homology is 286 (about twice as big as  for $\mathbb{Q}$-coefficients), and
the bidegrees of the multiplicative generators are equal to:
$$\deg(x_0)=q^2,~\deg(x_1)=q^4t^2,~\deg(x_2)=q^6t^4,~\deg(x_3)=q^8t^6,$$
$$\deg(x_4)=q^{10}t^8,~\deg(x_5)=q^{12}t^{10},~\deg(x_6)=q^{14}t^{12},$$
$$\deg(\xi_1)=q^6t^3,~\deg(\xi_3)=q^{10}t^7,~\deg(\xi_5)=q^{14}t^{11}.$$
 
Finally, one can check that the Khovanov homology of this knot has nontrivial $\mathbb{Z}_7$-torsion in degree $q^{20}t^{14}$.

\end{appendix}

Eugene Gorsky

Department of Mathematics

Stony Brook University

Stony Brook, NY 11794

{\tt egorsky@math.sunysb.edu}

\medskip

Alexei Oblomkov

Department of Mathematics

University of Massachusetts, Amherst

Amherst, MA 01003

{\tt oblomkov@math.umass.edu}

\medskip

Jacob Rasmussen

Department of Pure Mathematics

University of Cambridge

Centre for Mathematics Sciences

Wilberforce Road, CB3 0WB

United Kingdom 

{\tt J.Rasmussen@dpmms.cam.ac.uk }
 \end{document}